\documentclass[a4paper,12pt,reqno]{amsart}
\usepackage{marginnote}
\usepackage{amsmath}
\usepackage{mathtools}
\usepackage{amsfonts}
\usepackage{amscd}
\usepackage{amssymb}
\usepackage{graphicx}
\usepackage{mathrsfs}
\usepackage{dsfont}
\usepackage{tikz}

\usepackage[dvips,all,arc,curve,color,frame]{xy}
\usepackage{enumitem}

\usepackage[colorlinks]{hyperref}
\usepackage[nameinlink]{cleveref}
\usepackage{tikz,mathrsfs}
\usetikzlibrary{arrows,decorations.pathmorphing,decorations.pathreplacing,positioning,shapes.geometric,shapes.misc,decorations.markings,decorations.fractals,calc,patterns}

\newtheorem{thm}{Theorem}[section]
\newtheorem{prop}[thm]{Proposition}
\newtheorem{lemma}[thm]{Lemma}
\newtheorem{claim}[thm]{Claim}
\newtheorem{defin}[thm]{Definition}
\newtheorem{cor}[thm]{Corollary}

\newtheorem{obs}[thm]{Observation}
\newtheorem{str}[thm]{Strategy}

\theoremstyle{definition}

\begin{document}
\author{Vojt\v{e}ch Dvo\v{r}\'ak}
\address[Vojt\v{e}ch Dvo\v{r}\'ak]{Department of Pure Maths and Mathematical Statistics, University of Cambridge, UK}
\email[Vojt\v{e}ch Dvo\v{r}\'ak]{vd273@cam.ac.uk}

\title[Waiter-Client Triangle-Factor Game on the Edges of $K_n$]{Waiter-Client Triangle-Factor Game on the Edges of the Complete Graph
}


\begin{abstract} 
Consider the following game played by two players, called Waiter and Client, on the edges of $K_n$ (where $n$ is divisible by $3$). Initially, all the edges are unclaimed. In each round, Waiter picks two yet unclaimed edges. Client then chooses one of these two edges to be added to Waiter's graph and one to be added to Client's graph. Waiter wins if she forces Client to create a $K_3$-factor in Client's graph at some point, while if she does not manage to do that, Client wins. 

It is not difficult to see that for large enough $n$, Waiter has a winning strategy. The question considered by Clemens et al. is how long the game will last if Waiter aims to win as soon as possible, Client aims to delay her as much as possible, and both players play optimally. Denote this optimal number of rounds by $\tau_{WC}(\mathcal{F}_{n,K_3-\text{fac}},1 )  $. Clemens et al. proved that $\frac{13}{12}n \leq \tau_{WC}(\mathcal{F}_{n,K_3-\text{fac}},1 ) \leq \frac{7}{6}n+o(n) $, and conjectured that $\tau_{WC}(\mathcal{F}_{n,K_3-\text{fac}},1 ) = \frac{7}{6}n+o(n) $. In this note, we verify their conjecture.
\end{abstract}

\maketitle

\section{Introduction} 

\textit{Positional games} are a large class of two-player combinatorial games characterized by the following setting. We have a (usually finite) set $X$, called a \textit{board}; a family of subsets of $X$, called \textit{winning sets}; and a rule determining which player wins the game. These games attracted wide attention, starting with the papers of Hales and Jewett \cite{hales2009regularity} and Erd\H{o}s and Selfridge \cite{erd1973and}. We refer the reader interested in positional games to the book of Hefetz, Krivelevich, Stojakovi{\'c} and Szab{\'o} \cite{hefetz2014positional}.

Probably the most studied type of the positional games are the so-called \textit{Maker-Breaker games}. These games are played by two players, called Maker and Breaker, as follows. We are given an integer $b \geq 1$, a set $X$ and a family of winning sets $\mathcal{F}$ of the subsets of $X$. The players alternately claim previously unclaimed elements of $X$, with Maker claiming one element in each round and Breaker claiming $b$ elements in each round. Maker wins if he manages to claim any $F \in \mathcal{F}$, while Breaker wins if she prevents Maker from doing so. We are usually interested in the two main questions: which player has a winning strategy for given $X,b,\mathcal{F}$? And if Maker has a winning strategy, how fast can he win? Many variants of the Maker-Breaker games were considered, see for instance \cite{ben2012hitting,corsten2020odd,duchene2020maker,hefetz2009fast,muller2014threshold,nenadov2016threshold}. 

Another rather similar type of the well studied positional games are the \textit{Waiter-Client games}. These games are played by two players, called Waiter and Client, in the following manner. As in the Maker-Breaker games, we are given an integer $b \geq 1$, a set $X$ and a family of winning sets $\mathcal{F}$ of the subsets of $X$. In each round, Waiter picks $b+1$ previously unclaimed elements of $X$ and offers them to Client. Client chooses one of these elements and adds it to his graph, while the remaining $b$ elements become a part of Waiter's graph. Waiter wins if she forces Client to create a winning set $F \in \mathcal{F}$ in Client's graph. If Client can prevent that, he wins. Once again, the questions that interest us are: for given $X,b,\mathcal{F}$, which player has a winning strategy? And if Waiter can win, how fast can she guarantee her victory to be? These problems are well studied in many cases, see for instance \cite{bednarska2016manipulative,hefetz2016waiter,hefetz2017waiter,krivelevich2018waiter,tan2017waiter}. 



Assume that for our triple $X,b,\mathcal{F}$, Waiter wins the corresponding Waiter-Client game. Then we will denote by $\tau_{WC}(\mathcal{F},b)$ the number of rounds of the game when Waiter tries to win as fast as possible, Client tries to slow her down as much as possible, and they both play optimally. What the ground set $X$ is will usually be clear from the context.

When $b=1$, we call the corresponding Waiter-Client game \textit{unbiased}. Recently, Clemens et al. \cite{clemens2020fast} studied several unbiased Waiter-Client games played on the edges of the complete graph, i.e. with $X=E(K_n)$. For $n$ divisible by $3$, they considered the triangle-factor game, where the winning sets are the collections of $\frac{n}{3}$ vertex disjoint triangles. It is not hard to verify that for $n$ large enough, Waiter can win this game. Moreover, Clemens et al. obtained the following theorem giving the lower and upper bounds on the optimal duration of the game.

\begin{thm}\label{mainclemens}
Assume $n$ is divisible by $3$ and large enough that Waiter wins the corresponding unbiased triangle-factor game on the edges of $K_n$. Then $$\frac{13}{12}n \leq \tau_{WC}(\mathcal{F}_{n,K_3-\text{fac}},1 ) \leq \frac{7}{6}n+o(n). $$
\end{thm}

Further, they make a conjecture that $\tau_{WC}(\mathcal{F}_{n,K_3-\text{fac}},1 ) = \frac{7}{6}n+o(n)$. Our aim in this note is to improve the lower bound from $\frac{13}{12}n$ to $\frac{7}{6}n$ and hence to verify their conjecture.

\begin{thm}\label{major}
Assume $n$ is divisible by $3$ and large enough that Waiter wins the corresponding unbiased triangle-factor game on the edges of $K_n$. Then $$\tau_{WC}(\mathcal{F}_{n,K_3-\text{fac}},1 ) \geq \frac{7}{6}n.$$
\end{thm}


The rest of the paper is organized as follows:
we set up the necessary definitions and prove some easy results about these in Section \ref{section2}. After that, we prove Theorem \ref{major} in Section \ref{section3}, by giving a strategy of Client and analyzing the game when Client uses this strategy. 

\section{Good and bad connected components in the graph of Client}\label{section2}


We need the following characterization of the connected graphs that contain a triangle-factor, yet have few edges.

\begin{obs}\label{mincyclesetc}
Let $G$ be a connected graph with a triangle-factor and $|V(G)|=n_0$ (where $n_0$ clearly must be divisible by $3$). Then $|E(G)| \geq \frac{4}{3}n_0-1$. Moreover if $|E(G)| = \frac{4}{3}n_0-1$, the triangle-factor is unique, and $\frac{n_0}{3}$ triangles in this triangle-factor are the only cycles in $G$.
\end{obs}

\begin{proof}
We know that $G$ has at least one triangle-factor, consisting of the triangles $$T_1=a_1 b_1 c_1,...,T_{\frac{n_0}{3}}=a_{\frac{n_0}{3}} b_{\frac{n_0}{3}} c_{\frac{n_0}{3}}.$$ If $G$ contains multiple triangle-factors, pick one arbitrarily. Denote by $t$ the total number of edges in $G$ between the sets of the form $\lbrace a_i,b_i,c_i \rbrace$ and $\lbrace a_j,b_j,c_j \rbrace$ for $i \neq j$.

Now consider the auxiliary graph $H$, whose vertices are $T_1,...,T_{\frac{n_0}{3}}$ and where $T_i \sim T_j$ (for $i \neq j$) if there is at least one edge between $\lbrace a_i,b_i,c_i \rbrace$ and $\lbrace a_j,b_j,c_j \rbrace$. 

As $G$ is connected, it follows that $H$ must also be connected, which in particular implies $$t \geq |E(H)| \geq |V(H)|-1=\frac{n_0}{3}-1,$$ with equality if and only if $H$ is a tree and no two sets of the form $\lbrace a_i,b_i,c_i \rbrace$ and $\lbrace a_j,b_j,c_j \rbrace$ are connected by more than one edge in $G$.

That implies $G$ has at least $3 \times \frac{n_0}{3}=n_0$ edges within the sets $\lbrace a_i,b_i,c_i \rbrace$ (as each triangle contains three edges), and $t \geq \frac{n_0}{3}-1$ edges between the sets of the form $\lbrace a_i,b_i,c_i \rbrace$ and $\lbrace a_j,b_j,c_j \rbrace$. So overall, $$|E(G)| \geq n_0+t \geq \frac{4}{3}n_0-1,$$ as required.

Next, assume we have equality. Then we must have $t=\frac{n_0}{3}-1.$ If $G$ contained any cycle not of the form $a_i b_i c_i$, this cycle would contain two consecutive vertices $v_1, v_2$ with $v_1 \sim v_2$ in $G$ and $v_1,v_2$ not being in the same triangle of our triangle-factor.  As noted before, $t=\frac{n_0}{3}-1$ implies that $H$ is a tree and that no two sets of the form $\lbrace a_i,b_i,c_i \rbrace$ and $\lbrace a_j,b_j,c_j \rbrace$ are connected by more than one edge in $G$. But that in particular implies that every path between $v_1$ and $v_2$ in $G$ uses the edge $v_1 v_2$, which is a contradiction to $v_1,v_2$ being part of the same cycle in $G$. 

Hence, in the case of equality, only cycles in $G$ are the ones of the form $a_i b_i c_i$ for $1 \leq i \leq \frac{n_0}{3}$.
\end{proof}





Taking into account Observation \ref{mincyclesetc}, the following definition is rather natural.

\begin{defin}\label{impodef}
Consider the connected components of Client's graph. We will make a distinction between good and bad ones. When a new connected component is created in Client's graph, initially it is called bad. Now assume that Client adds the edge $ab$ to his graph. Then we update the state of the connected component of Client's graph containing the edge $ab$ as follows.
\begin{itemize}
    \item If both $a $ and $b$ are already a part of good connected components of Client's graph (whether of the same one or of different ones), we consider the connected component of Client's graph containing the edge $ab$ to be a good component.
    \item If the edge $ab$ connects good and bad components of Client's graph together, or adds a new vertex to a good component, we consider the connected component containing the edge $ab$ to be a good component.
    \item If neither of the vertices $a,b$ was a part of a good connected component before the edge $ab$ was added, but after adding the edge $ab$, the connected component $K$ of Client's graph containing the edge $ab$ has a triangle-factor and satisfies $|E(K)|=\frac{4}{3}|V(K)|-1$, we consider the connected component containing the edge $ab$ to be a good component. Moreover, in this case (and only in this case), we say that this connected component of Client's graph was declared to be good at the time when we added the edge $ab$ (or that a new good connected component of Client's graph was created).
    \item In any other case, the connected component of Client's graph containing the edge $ab$ is bad.
\end{itemize}
\end{defin}



\begin{obs}\label{keyobswaiter}
If Waiter won the game, and throughout the game, at most $\frac{n}{6}$ connected components of Client's graph were declared to be good, then final Client's graph contains at least $\frac{7}{6}n$ edges.
\end{obs}

\begin{proof}
Call $C_1,...,C_k$ the connected components of the final graph of Client that are good, and $C_{k+1},...,C_l$ the connected components of the final graph of Client that are bad. Then if $1 \leq i \leq k$, we have $|E(C_i)| \geq \frac{4}{3}|V(C_i)|-1$, while if $k+1 \leq i \leq l$, we have $|E(C_i)| \geq \frac{4}{3}|V(C_i)|$. 

As $k \leq \frac{n}{6}$, we get $$ \sum_{i=1}^{l} |E(C_i)| \geq \frac{4}{3}  \sum_{i=1}^{l} |V(C_i)|-k=\frac{4}{3}n-k \geq \frac{7}{6}n ,$$ as required.
\end{proof}

\section{Proof of Theorem \ref{major}}\label{section3}

Throughout this section, denote by $G_i$ Client's graph after $i$ rounds.

We start with the definition that will be used when describing Client's strategy later. 

\begin{defin}\label{crudef}
Edge that Waiter offers to Client is called \textit{crucial} if by choosing it to Client's graph, Client would create a new good connected component. 
\end{defin}


Using Definition \ref{impodef}, it is trivial to check that it must be the case that the two endpoints of any crucial edge are in the same connected component of Client's graph at the time it is offered. Hence, we can make the following further definition.

\begin{defin}
If Client is offered crucial edge $ab$ in the $i$th round, we say that the connected component of $G_{i-1}$ containing $a b$ is \textit{crucial in the $i$th round}. 
\end{defin}

Client will pick the edges according to the following strategy.

\begin{str}\label{strat}
Client considers the following two possibilities.
\begin{itemize}
    \item If one of the two edges that he is offered is crucial, while the other edge offered is not crucial, he picks to Client's graph the edge that is not crucial. 
    \item In any other case, he picks one edge arbitrarily. 
\end{itemize}
\end{str}

Moreover, we call any round when Client is offered two crucial edges \textit{difficult}.

In this section, we will work towards the following result.

\begin{prop}\label{whatweinfactshow}
If Client plays according to Strategy \ref{strat}, he can ensure that, throughout the game, at most $\frac{n}{6}$ good connected components were created. 
\end{prop}

Theorem \ref{major} then follows by applying Observation \ref{keyobswaiter}.

We start by proving an easy lemma.

\begin{lemma}\label{easytech}
Assume that $C$ is a bad connected component of Client's graph. Then there exists at most one (yet unclaimed) edge with its endpoints in $C$ that, if offered by Waiter to Client, would be crucial.
\end{lemma}

\begin{proof} Assume some such edge exists. We will show it is unique.

If every vertex of $C$ is already in some triangle of Client's graph, then by Definition \ref{impodef} we can't have a crucial edge for $C$. If more than three vertices of $C$ are not in any triangle of Client's graph, we can't have a crucial edge for $C$ either, since by adding just one edge we can't create a triangle-factor of $C$. Finally, if precisely one or two vertices of $C$ are not in any triangle of Client's graph yet, we can't have a crucial edge for $C$. Since if we did and this edge created some new triangles in $C$ (which it must), at least one vertex of $C$ would be in at least two triangles after adding this edge, contradicting Observation \ref{mincyclesetc}.

So precisely three vertices $x_1,x_2,x_3$ of $C$ are not in any triangle of Client's graph yet. By Observation \ref{mincyclesetc}, any edge we add into $C$ that will make it into a good connected component must complete a triangle $x_1 x_2 x_3$. But then it must be the case that the triangle $x_1 x_2 x_3$ misses precisely one edge in $C$, and this missing edge thus must be our crucial edge.
\end{proof}

The heart of our proof of Proposition \ref{whatweinfactshow} is Lemma \ref{essential1}. Before stating it, we need the following definition.

\begin{defin}
Let $A_i$ be the subset of $V(K_n)$ consisting of the vertex sets of the connected components of $G_{i-1}$ that were crucial in the $i$th round. Let $B_i = A_i \setminus \bigcup_{j=1}^{i-1} A_j$.
\end{defin}

We emphasize that for $v \in A_i$, we genuinely need Waiter to offer Client in the $i$th round some crucial edge $ab$ with the endpoints $a,b$ being in the same connected component of $G_{i-1}$ as $v$, and that it is not enough if such an edge simply exists but Waiter does not offer it to Client in the $i$th round.

\begin{lemma}\label{essential1}
Assume that Waiter offered Client some crucial edge in the $i$th round, and let $C$ be the crucial connected component of $G_{i-1}$ containing the endpoints of this edge. Then $|V(C) \cap B_i| \geq 3$.
\end{lemma}

\begin{proof}
Let $C_1,...,C_r \subset V(C)$ be the following subsets. We include in this collection any $C_j \subset V(C)$ that at some point was a vertex set of the crucial connected component of Client's graph (note that for all of these, Client did not choose the corresponding crucial edge to his graph though, and instead it went to the graph of Waiter - else $C$ would be good by Definition \ref{impodef}).

Next, let $D_1,...,D_s \subset V(C)$ be the same collection as $C_1,...,C_r$, except that we delete any $C_j$ for which we can find $k \neq j$ with $C_j \subset C_k$.

As the proof of Lemma \ref{essential1} is rather long, we shall have several claims throughout to keep the structure of the proof of Lemma \ref{essential1} as clear as possible.

\begin{claim}\label{cm1}
$D_1,...,D_s$ are disjoint subsets of $V(C)$.
\end{claim}

\begin{proof}[Proof of Claim \ref{cm1}]
Assume $v \in D_j \cap D_k$ for some $v \in V(C)$ and some $1 \leq j, k \leq s$ with $j \neq k$. Assume also that Waiter offered the crucial edge for $D_j$ before offering the crucial edge for $D_k$ (she clearly could not have offered both at the same time, as then $D_j,D_k$ would be the same set because their intersection is non-empty). Let $w$ be any other vertex of $D_j$. Since $v,w$ were in the same connected component of Client's graph at the time when $D_j$ was a vertex set of a crucial component, they will stay in the same connected component forever after, and in particular as $v \in D_k$, we also have $w \in D_k$. As $w$ was arbitrary, that implies $D_j \subset D_k$, which is a contradiction to our definition of the sets $D_1,...,D_s$.
\end{proof}

Let $$X=V(C) \setminus \bigcup_{j=1}^{s} D_j.$$ Then clearly $X \subset V(C) \cap B_i$. Also, by the definition of a crucial connected component and by Claim \ref{cm1}, both $|V(C)|$ and $|\bigcup_{j=1}^{s} D_j|=\sum_{j=1}^{s}|D_j|$ are divisible by $3$. So if we can show that $|X|>0$, that immediately implies that $|X| \geq 3$ and proves Lemma \ref{essential1}. 

Assume for contradiction that $X= \emptyset$ and hence $V(C)=\bigcup_{j=1}^{s} D_j$. First we rule out the case $s=1$.

\begin{claim}\label{cm3}
We have $s \geq 2$.
\end{claim}

\begin{proof}[Proof of Claim \ref{cm3}]
Assume for contradiction that we have $s=1$ and $V(C)=D_1$. At the first time that $V(C)$ was a vertex set of a crucial connected component $C_0$ of Client's graph, by Observation \ref{mincyclesetc} we must have had $|E(C_0)|=\frac{4}{3}|V(C)|-2$. But by Lemma \ref{easytech}, the crucial edge for $C_0$ was unique, and hence now we must have $$|E(C)| \geq |E(C_0)|+1=\frac{4}{3}|V(C)|-1.$$
But that contradicts $C$ being a crucial connected component of Client's graph.
\end{proof}

For $j=1,...,s$, let $p_j q_j$ be the crucial edge offered when we had a crucial connected component with a vertex set $D_j$, and let $r_j$ be the vertex that $p_j,q_j$ would have formed a triangle with in Client's graph at that time, had Client taken $p_j q_j$ to his graph. Let $C'$ be a connected component we obtain if Client picks a crucial edge with the endpoints in $C$ to his graph in the $i$th round. Clearly $V(C)=V(C')$.

For $v \in V(C)$ belonging to some $D_j$, denote by $T(v)$ the set of all the sets $D_k$ with $k \neq j$ that $v$ is connected to in $C'$.

\begin{claim}\label{cm2}
Take $p_j,q_j$ for any $1 \leq j \leq s$. Then $T(p_j),T(q_j) \neq \emptyset$ and moreover $T(p_j) \cap T(q_j)= \emptyset$.
\end{claim}

\begin{proof}[Proof of Claim \ref{cm2}]
We know that $p_j$ is a vertex of a triangle $p_j v_1 v_2$ in $C'$, for some $v_1,v_2 \in V(C)$. We know that $q_j \notin \lbrace v_1,v_2 \rbrace $, since the edge $p_j q_j$ belongs to Waiter's graph (as it was offered previously to Client, but Client did not take it). But we also know $$\lbrace v_1,v_2 \rbrace \cap D_j \subset \lbrace q_j,r_j \rbrace,$$ since any other vertex of $D_j$ was in some triangle in Client's graph already at the time when $D_j$ was a vertex set of a crucial connected component, and by Observation \ref{mincyclesetc} every vertex of $C'$ is in precisely one triangle in Client's graph. Hence there must exist some $k \neq j$ with $\lbrace v_1,v_2 \rbrace \cap D_k \neq \emptyset$, and $T(p_j) \neq \emptyset$ follows.

We derive $T(q_j) \neq \emptyset$ analogously.

Finally, assume for contradiction that $T(p_j) \cap T(q_j) \neq \emptyset$. Take $D_k$ such that $D_k \in T(p_j) \cap T(q_j)$. Then we have $w_1,w_2 \in D_k$ such that $p_j \sim w_1$ and $q_j \sim w_2$ in $C'$ (it may happen that $w_1,w_2$ are the same vertex). Let $w_1 z_1 \dots z_u w_2$ be a path between $w_1$ and $w_2$ in $D_k$. Then $p_j w_1 z_1 \dots z_u w_2 q_j r_j$ is a cycle of length at least four in $C'$, which by Observation \ref{mincyclesetc} gives a contradiction to $C'$ being a good connected component.
\end{proof}




Now consider $$S_0= \lbrace ab: ab \in E(C'); \, \, a,b \text{ are in the different sets of the form } D_j \rbrace.$$ 
We will modify this set repeatedly as follows. As long as $S_k$ contains any edge $ab$ such that there are also both some edge $ab'$ for $b' \neq b$ and some edge $a'b$ for $a' \neq  a$ in $S_k$, erase some such edge $ab$ from $S_k$ to form the set $S_{k+1}$. This process eventually terminates with some final set $S_{\text{final}}$. Write $S=S_{\text{final}}$.

Let $I$ be the following auxiliary graph. Its vertices are $D_1,...,D_s$ and $D_j \sim D_k$ in $I$ if there is at least one edge going between $D_j$ and $D_k$ in $S$. 

\begin{claim}\label{cm4}
The minimum degree of $I$ satisfies $\delta(I) \geq 2$.
\end{claim}

\begin{proof}[Proof of Claim \ref{cm4}]
Consider any $j$, $1 \leq j \leq s$. By Claim \ref{cm2}, $S_0$ contained at least one edge of the form $p_j a$ and at least one edge of the form $q_j b$ for some $a,b \in V(C)$. Moreover, by the definition of the process by which we obtained $S$ from $S_0$, we can never erase the last edge of either of these two forms from $S_k$ when creating $S_{k+1}$. Hence $S$ also contains at least one edge of the form $p_j a'$ and at least one edge of the form $q_j b'$ for some $a',b' \in V(C)$. 
Finally, by Claim \ref{cm2}, $a'$ and $b'$ can not be in the same set $D_k$, giving $d_I(D_j) \geq 2$. As $j$ was arbitrary, the result follows.
\end{proof}

 But Claim \ref{cm4} in particular implies that $I$ contains some cycle. That in turn implies that $C'$ contains some cycle which contains at least three edges from $S$. But by the construction of $S$, any three different edges of $S$ have at least four different endpoints in total. So $C'$ contains a cycle of length at least four, contradicting Observation \ref{mincyclesetc}.

Thus the proof of Lemma \ref{essential1} is finished.
\end{proof}


\begin{cor}\label{corwaiter}
If $i$ is a difficult round, then $|B_i| \geq 6 $.
\end{cor}

\begin{proof}
By Lemma \ref{easytech}, we know that the two crucial edges offered in the difficult round could not have been in the same connected component of $G_{i-1}$. Result then follows by applying Lemma \ref{essential1}.
\end{proof}

We are now ready to prove Proposition \ref{whatweinfactshow}, and hence as discussed before also complete the proof of Theorem \ref{major}.

\begin{proof}[Proof of Proposition \ref{whatweinfactshow}]
Assume Client creates $k$ good connected components in his graph throughout the game, and that overall the game lasted $T$ rounds (where $T \geq k$, since Client can create at most one good connected component in each round). That implies there were at least $k$ difficult rounds, as Client does not create good connected components in any other round. But then using Corollary \ref{corwaiter} and the fact that $B_i \cap B_j = \emptyset$ whenever $i \neq j$, we get $$n \geq |\bigcup_{i=1}^{T} B_i|=\sum_{i=1}^{T} |B_i| \geq 6k.$$

It follows that $k \leq \frac{n}{6}$, as required.
\end{proof}

This resolves the triangle-factor game fully. As suggested by Clemens et al. \cite{clemens2020fast}, it may also be interesting to consider the general $K_k$-factor game instead of just the case $k=3$. It is not hard to show that Waiter still wins this game, but non-trivial lower and upper bounds for $\tau_{WC}(\mathcal{F}_{n,K_k-\text{fac}},1 )$ would definitely be of interest.

\section*{Acknowledgements}

The author would like to thank his PhD supervisor professor B\'ela Bollob\'as for his support.  

\bibliographystyle{abbrv}
\bibliography{sample}

\end{document}